\documentclass[11pt,leqno]{amsart}
\usepackage{a4wide}
\usepackage{amssymb,upgreek}
\usepackage{amsmath}
\usepackage{amsthm}
\usepackage{mathrsfs}
\usepackage{bm,euscript,csquotes,comment}
\usepackage[curve]{xypic}
\usepackage[usenames,dvipsnames]{xcolor}
\usepackage{hyperref}
\theoremstyle{plain}

\makeatletter
\def\cal@symb#1|#2{\expandafter\def\csname #2#1\endcsname{\mathcal{#1}}}
\def\calsymbols#1#2{\@for\@tmpz:=#2\do{\expandafter\cal@symb\@tmpz|{#1}}}

\def\dmth@p#1|{\expandafter\let\csname#1\endcsname\relax
  \expandafter\DeclareMathOperator\csname#1\endcsname{#1}}
\def\operators#1{\@for\@tmpz:=#1\do{\expandafter\dmth@p\@tmpz|}}
\makeatother
\calsymbols{c}{A,B,C,D,E,F,G,H,I,J,K,L,M,N,O,P,Q,R,S,T,U,V,W,X,Y,Z}
\operators{Ab,tor,Op,Mod,Vect,ind,res,End,im,Hom,act,Stab,Res,cores,val}

\newcommand{\chara}{\operatorname{char}}

\newtheorem{theorem}{Theorem}[section]
\newtheorem{corollary}[theorem]{Corollary}
\newtheorem{lemma}[theorem]{Lemma}
\newtheorem{proposition}[theorem]{Proposition}

\newtheorem{definition}[theorem]{Definition}

\newtheorem{example}[theorem]{Example}

\theoremstyle{definition}
\newtheorem{remark}[theorem]{Remark}

\title{\textbf{The Bernstein center in natural characteristic}}
\author{Konstantin Ardakov, Peter Schneider}
\date{\today}

\address{Mathematical Institute, University of Oxford, Oxford OX2 6GG}
\email{ardakov@maths.ox.ac.uk}
\urladdr{http://people.maths.ox.ac.uk/ardakov/}

\address{ Universit\"at M\"unster,  Mathematisches Institut,  Einsteinstr. 62, 48291 M\"unster, Germany}
\email{pschnei@wwu.de}
\urladdr{https://www.uni-muenster.de/Arithm/schneider/}

\begin{document}

\begin{abstract}
Let $G$ be a locally profinite group and let $k$ be a field of positive characteristic $p$. Let $Z(G)$ denote the center of $G$ and let $\mathfrak{Z}(G)$ denote the Bernstein center of $G$, that is, the $k$-algebra of natural endomorphisms of the identity functor on the category of smooth $k$-linear representations of $G$. We show that if $G$ contains an open pro-$p$ subgroup but no proper open centralisers, then there is a natural isomorphism of $k$-algebras $\mathfrak{Z}(Z(G)) \xrightarrow{\cong} \mathfrak{Z}(G)$. We also describe $\mathfrak{Z}(Z(G))$ explicitly as a particular completion of the abstract group ring $k[Z(G)]$. Both conditions on $G$ are satisfied whenever $G$ is the group of points of any connected smooth algebraic group defined over a local field of residue characteristic $p$. In particular, when the algebraic group is semisimple, we show that $\mathfrak{Z}(G) = k[Z(G)]$.
\end{abstract}

\maketitle

\section{Introduction}

Let $G$ be a locally profinite group and let $k$ be any field. Recall that a $k$-linear representation $V$ of $G$ is said to be \emph{smooth} if every vector in $V$ is fixed by an open subgroup of $G$. The smooth $k$-linear representations of $G$ form an abelian category $\Mod(G)$ and the \emph{Bernstein center} $\mathfrak{Z}(G)$ of $\Mod(G)$ is by definition the ring of natural endomorphisms of the identity functor on $\Mod(G)$: it is naturally a commutative $k$-algebra which acts on every $V$ in $\Mod(G)$ by $k$-linear endomorphisms commuting with the action of $G$ of $V$.

When $G$ is a reductive group over a local non-archimedean field of mixed characteristic $(0,p)$ and $k$ is the complex numbers, the Bernstein center $\mathfrak{Z}(G)$ was studied in detail by Bernstein in \cite{Ber}. He found that $\mathfrak{Z}(G)$ decomposes as a direct product of smaller algebras that are now called the \emph{Bernstein components}, and he found a parametrisation of these components in terms of $G$-conjugacy classes of cuspidal pairs on the Levi subgroups of $G$. Consequently, in this classical case, the Bernstein center is rather large, and it plays a fundamental role in the classical local Langlands correspondence \cite{Hen}.

Recently, there has been a lot of interest in a possible mod-$p$ local Langlands correspondence, where instead one works with smooth representations defined over a field $k$ of characteristic $p$; see, for example, \cite{Vig04}, \cite{BP}, \cite{Sch} and \cite{BH+}.  It is therefore natural to enquire about the structure of $\mathfrak{Z}(G)$. There has been an expectation that in this case $\mathfrak{Z}(G)$ is small, which we will confirm in this paper by a precise explicit calculation.

We begin the description of our results by introducing the following completion of the group algebra $k[Z(G)]$ of the center $Z(G)$ of $G$:
\[ \widehat{k[Z(G)]} := \varprojlim \, k[Z(G)/Z']\]
where $Z'$ runs over all compact open subgroups of $Z(G)$.  By considering the natural action of $k[Z(G)]$ on the subspace $V^U$ of $U$-fixed points of a given object $V$ in $\Mod(G)$ as $U$ ranges over all compact open subgroups of $G$, it is not difficult to see that the algebra $\widehat{k[Z(G)]}$ acts by $k[G]$-linear endomorphisms on every $V$ in $\Mod(G)$ in a natural way. This defines a $k$-algebra homomorphism $\Phi_G : \widehat{k[Z(G)]} \to \mathfrak{Z}(G)$ which is easily seen to be injective. Using this map we can now state our main result.

\begin{theorem}\label{thm:MainIntro}
Let $G$ be a locally profinite group that contains an open pro-$p$ subgroup and let $k$ be a field of characteristic $p$. Suppose that $G$ contains no proper open centralisers. Then the natural map $\Phi_G : \widehat{k[Z(G)]} \xrightarrow{\cong} \mathfrak{Z}(G)$ is an isomorphism.
\end{theorem}

We show in Thm.\ \ref{thm:open-cent} below that the assumptions on $G$ are satisfied whenever $G=\mathbf{G}(\mathfrak{F})$ for some connected smooth algebraic group $\mathbf{G}$ defined over a local nonarchimedean field $\mathfrak{F}$ of residue characteristic $p$. By applying Theorem \ref{thm:MainIntro} to $G$ and to $Z(G)$ in turn, we obtain the following

\begin{corollary}
Under the assumptions of Theorem \ref{thm:MainIntro}, the natural map
\[\Phi_G \circ \Phi_Z^{-1} : \mathfrak{Z}(Z(G)) \xrightarrow{\;\cong\;} \mathfrak{Z}(G)\]
is an isomorphism of $k$-algebras.
\end{corollary}

Here is a brief summary of the contents of this paper. In $\S \ref{sec:Bcenter}$, we use a result of Positselski \cite{Pos} to show that $\mathfrak{Z}(G)$ is naturally isomorphic to the $k$-algebra of bi-equivariant $k$-linear endomorphisms $\End_{\Mod(G \times G)}(C^\infty_c(G,k))$ of the space $C^\infty_c(G,k)$ of compactly supported locally constant $k$-valued functions on $G$. In $\S \ref{sec:univ-endo-ring}$, we use the fact that $C^\infty_c(G,k)$ is isomorphic in $\Mod(G)$ to the compact induction from $U$ to $G$ of the corresponding space of locally constant $k$-valued functions on some fixed open pro-$p$ subgroup $U$ of $G$, together with a version of the Mackey decomposition, to reduce the calculation of the bi-equivariant endomorphism algebra to the calculation of the fixed points in some profinite permutation modules $k[[U]]$ under a twisted conjugation action of various open subgroups $U_w$ of $U$. In $\S \ref{sec:FixPtsPMM}$ we crucially use the hypothesis that $k$ has characteristic $p$ to show that in this situation, all fixed points in such profinite permutation modules arise from the \emph{finite} orbits. In $\S \ref{sec:twisted-action}$ we use the hypothesis on open centralisers in $G$ to show that there are no such finite $U_w$-orbits in $U$ --- unless $w$ happens to be central in $G$ --- and deduce Theorem \ref{thm:MainIntro} which reappears as Theorem \ref{thm:Main}. We also explain how to verify the hypothesis on open centralisers in the case of connected algebraic groups. Finally, in \S \ref{sec:PosThmNewPf} we give a self-contained proof of Positselski's Theorem, for the convenience of the reader unfamiliar with the categorical machinery in \cite{Pos}.

We thank F.\ Pop and M.-F.\ Vign\'eras for pointing out that Prop.\ \ref{prop:density} is known to the experts and for indicating an argument. After this paper was finished, A.\ Dotto has informed us that he had proved (using different methods) that for certain $p$-adic reductive groups, the Bernstein center is a local ring when one fixes a central character.

The second author acknowledges support from Deutsche Forschungsgemeinschaft (DFG, German Research Foundation) under Germany's Excellence Strategy EXC 2044 –390685587, Mathematics Münster: Dynamics–Geometry–Structure.

\section{Notation}

Throughout $G$ is a locally profinite group, and $k$ is any field. If $G$ contains an open pro-$p$ subgroup then we call $G$ locally pro-$p$. By $\Op_c(G)$ we denote the set of open compact subgroups of $G$. Furthermore, $\Mod(G)$, resp.\ $\Vect$, denotes the abelian category of all smooth $G$-representations in $k$-vector spaces, resp.\ of all $k$-vector spaces. As usual, $Z(G)$ is the center of $G$. Similarly, $Z(R)$, for a ring $R$, is the center of $R$. For any topological space $X$ we let $C_c^\infty(X,k)$ denote the $k$-vector space of $k$-valued locally constant functions with compact support on $X$. If $X$ is compact, resp.\ discrete, we simply write $C^\infty(X,k)$, resp.\ $C_c(X,k)$. For a compact open subset $C \subseteq X$ let $\chara_C \in C_c^\infty(X,k)$ denote the characteristic function of $C$. For any group $\Gamma$ acting on a set $Y$ and any subset $P$ of $Y$, we write $\Gamma_P$ to denote the stabiliser of $P$ in $\Gamma$.

\section{The Bernstein center of $\Mod(G)$}\label{sec:Bcenter}

The Bernstein center $\mathfrak{Z}(G)$ of the abelian category $\Mod(G)$ by definition is the ring of natural endomorphisms of the identity functor on $\Mod(G)$. It obviously is a commutative $k$-algebra, which acts functorially on any object in $\Mod(G)$. In order to describe it in more down to earth terms we observe that the group $G \times G$ acts smoothly on $C_c^\infty(G,k)$ by ${^{(g_1,g_2)} F}(-) := F(g_1^{-1} - g_2)$. In the following we write $G_\ell$, resp.\ $G_r$, and correspondingly $g_\ell$ and $g_r$, if we refer to the action of the left, resp.\ right, factor $G$, and correspondingly $g \in G$.

We first make the larger $k$-algebra $\mathfrak{A}(G)$ of all natural endomorphisms of the forgetful functor $\Mod(G) \rightarrow \Vect$ more explicit. An element of $\mathfrak{A}(G)$, of course, is a family of $k$-linear endomorphisms $T_V$ for any $V$ in $\Mod(G)$ which commute, in an obvious sense, with any morphism in $\Mod(G)$. Any $g \in G$ via its action on the $V$ can be seen as an element $(g_V)_V \in \mathfrak{A}(G)$. Viewing $C_c^\infty(G,k)$ as an object in $\Mod(G)$ via the $G_\ell$-action we consider the $k$-algebra homomorphism
\begin{align*}
  \mathfrak{A}(G) & \longrightarrow \End_k(C_c^\infty(G,k)) \\
          (T_V)_V & \longmapsto T_{C_c^\infty(G,k)} \ .
\end{align*}
Observing that the $G_r$-action on $C_c^\infty(G,k)$ is by $G_\ell$-equivariant maps we see that any $T_{C_c^\infty(G,k)}$ has to commute with this $G_r$-action. The above map therefore is, in fact, a map
\begin{equation*}
  \Theta : \mathfrak{A}(G) \longrightarrow \End_{\Mod(G_r)}(C_c^\infty(G,k)) \ .
\end{equation*}
We then have the following result due to Positselski (\cite{Pos} Prop.\ 3.6(a)). See section \ref{sec:PosThmNewPf} for an alternative proof.

\begin{proposition}\label{prop:Pos}
   The map $\Theta$ is an isomorphism.
\end{proposition}

In fact we will only need the injectivity of $\Theta$ the proof of which is easy.

Obviously $\mathfrak{Z}(G) = \{(T_V)_V \in \mathfrak{A}(G) : (T_V \circ g_V)_V = (g_V \circ T_V)_V \; \text{for any $g \in G$}\}$. But as a consequence of Prop.\ \ref{prop:Pos} the condition on the right hand side only needs to be verified on $C_c^\infty(G,k)$. This implies the following description.

\begin{corollary}\label{cor:B}
   $\mathfrak{Z}(G) = \End_{\Mod(G \times G)}(C_c^\infty(G,k))$.
\end{corollary}

We will now write down some obvious elements in $\mathfrak{Z}(G)$.

\begin{definition}
Let $Z$ be a locally profinite abelian group,  and let $\widehat{k[Z]}$ denote the following completion of $k[Z]$:
\[ \widehat{k[Z]} := \varprojlim_{Z' \in \Op_c(Z)} \, k[Z/Z'].\]
\end{definition}

To get some feeling for this completion, consider the following special case.

\begin{remark}\label{rem:hkZdesc}
  Let $Z$ be a locally profinite abelian group. Suppose there is a compact open subgroup $Z_0$ of $Z$ and a discrete subgroup $A$ of $Z$ such that $Z = A \times Z_0$. Then $\widehat{k[Z]}$ is isomorphic to the usual completed group ring of the profinite group $Z_0$ but with coefficients taken in the discrete $k$-algebra $k[A]$:
\begin{equation*}
  \widehat{k[Z]}  \cong  k[A] [[Z_0]]  := \varprojlim_{Z_1 \in \Op_c(Z_0)} k[A][Z_0/Z_1]
\end{equation*}
If $A$ is finite this simplifies to $\widehat{k[Z]} = k[A] \otimes_k k[[Z_0]]$.
\end{remark}

\begin{example}
Suppose that $Z = \mathfrak{F}^\times$ for some finite extension $\mathfrak{F}$ of $\mathbb{Q}_p$. Fix a uniformiser $\pi \in o_{\mathfrak{F}}$, let $A$ be the subgroup of $Z$ generated by $\pi$ and let $Z_0 := o_{\mathfrak{F}}^\times$. Then $Z, A$ and $Z_0$ satisfy the conditions of Remark \ref{rem:hkZdesc}. Therefore in this case, $\widehat{k[\mathfrak{F}^\times]}$ is the completed group ring of the compact abelian $p$-adic Lie group $o_{\mathfrak{F}}^\times$ with coefficients in the Laurent polynomial ring $k[A] \cong k[\pi, \pi^{-1}]$ in the variable $\pi$.
\end{example}

\begin{lemma}\label{lem:k[[ZG]]action}
   $\widehat{k[Z(G)]}$ acts naturally on every $V \in \Mod(G)$ by $k[G]$-linear endomorphisms.
\end{lemma}
\begin{proof} For every compact open subgroup $H$ of $G$, restrict the $Z := Z(G)$-action on $V$ to $V^H = \{v \in V : h \cdot v = v$ for all $h \in H\}$. Since the action of $Z$ on $V$ commutes with the action of $H$, $V^H$ is a $Z$-submodule of $V$. The normal subgroup $Z \cap H$ of $Z$ acts trivially on $V^H$ and therefore the $Z$-action on $V^H$ factors descends to a well-defined action of $Z / ( Z \cap H )$ on $V^H$. By the definition of the completion $\widehat{k[Z]}$, there is a canonical map from this completion to $k[ Z / ( Z \cap H ) ]$, because $Z \cap H$ is a compact open subgroup of $Z$. In this way, every $V^H$ is naturally a $\widehat{k[Z]}$-module, via this homomorphism.  Let $J$ be another compact open subgroup of $G$, contained in $H$. Then $V^H$ is contained in $V^J$, and both are $\widehat{k[Z]}$-modules as explained above. We observe that the inclusion map $V^H \to V^J$ is in fact $\widehat{k[Z]}$-linear.

We have now defined a $\widehat{k[Z]}$-action on $V^H$ for every compact open subgroup $H$ of $G$ and we have checked that these actions are compatible with the inclusions $V^H \to V^J$ whenever $J \subseteq H$ is a smaller compact open subgroup. Therefore there is a well-defined action of $\widehat{k[Z]}$ on $V = \bigcup_H V^H$.

It remains to check that the $\widehat{k[Z]}$-action we have constructed commutes with the given $G$-action on $V$. Fix $x \in \widehat{k[Z]}$, $g \in G$ and $v \in V$. Choose a compact open subgroup $H$ of $G$ such that $v \in V^H$; then $v$ and $g\cdot v$ both lie in $V^J$ where $J := H \cap gHg^{-1} \in \Op_c(G)$. The action of $x$ on $V^J$ is equal to the action of its image $\overline{x}$ in $k[ZJ/J]$. Now $zJ \cdot (g \cdot v) = g \cdot (zJ \cdot v)$ for every $zJ \in ZJ/J$ so $x \cdot (g \cdot v) = \overline{x} \cdot (g \cdot v) = g \cdot (\overline{x} \cdot v) = g \cdot (x \cdot v)$ as required.
\end{proof}
Let $x \mapsto \Phi_V(x) \in \End_{\Mod(G)}(V)$ denote the action of $x \in \widehat{k[Z(G)]}$ on $V \in \Mod(G)$ that was constructed in Lemma \ref{lem:k[[ZG]]action}.
\begin{lemma}\label{lem:PhiMap} The map $x \mapsto \Phi(x) := (\Phi_V(x))_V$ is a $k$-algebra homomorphism
\[\Phi : \widehat{k[Z(G)]} \to \mathfrak{Z}(G).\]
\end{lemma}
\begin{proof} The action of $k[Z(G)]$ commutes with every morphism $\varphi : V \to W$ in $\Mod(G)$. This implies that $\Phi_W(x) \circ \varphi = \varphi \circ \Phi_V(x)$ for every $x \in \widehat{k[Z(G)]}$. So $(\Phi_V(x))_V$ does define an element $\Phi(x)$ in $\mathfrak{Z}(G)$. The verification that $\Phi$ is a $k$-algebra homomorphism is straightforward.
\end{proof}

\begin{proposition}\label{prop:Phi_C}
Let $C := C_c^\infty(G,k)$, regarded as an object in $\Mod(G_\ell)$. Then
\begin{enumerate}
\item $\Phi_C : \widehat{k[Z(G)]} \to \End_{\Mod(G_\ell)}(C)$ is injective, and
\item $\im(\Phi_C) \subseteq \End_{\Mod(G \times G)}(C)$.
\end{enumerate}
\end{proposition}
\begin{proof}
(1) Write $Z := Z(G)$. It is enough to show that $k[Z / (Z\cap U)]$ acts faithfully on $C^U = C^\infty_c(G,k)^{U_\ell} = C_c(U\backslash G,k)$ for every compact open subgroup $U$ of $G$. For any $x \in G$, consider the characteristic function $\chara_{Ux} \in C_c(U \backslash G,k)$ of the coset $Ux \in U \backslash G$; if $z \in Z$ then ${}^{z_\ell} \chara_{U} = \chara_{zU} = \chara_{Uz}$ shows that $\chara_U$ generates a $k[Z / (Z \cap U)]$-submodule of $C^U$ which is free of rank one. Hence $k[Z / (Z \cap U)]$ acts faithfully on $C^U$.

(2) This follows from the construction of $\Phi_C: \widehat{k[Z]} \to \End_{\Mod(G_\ell)}(C)$: it suffices to see that $\Phi_C(k[Z])$ commutes with the action of $G_r$, but $\Phi_C(z)(f) = {^{z_\ell} f}$ for any $f \in C$ and any $z \in Z$, and  $z_\ell$ commutes with $G_r$.
\end{proof}

We now can easily compute the Bernstein center for commutative groups.

\begin{proposition}\label{prop:commutative}
   If $G = Z$ is commutative then the maps
\begin{equation*}
  \widehat{k[Z]} \xrightarrow[\cong]{\;\Phi\;} \mathfrak{Z}(Z) \xrightarrow[\cong]{\;\Theta\;} \End_{\Mod(Z)}(C_c^\infty(Z,k))
\end{equation*}
are isomorphisms.
\end{proposition}
\begin{proof}
We already know  from Cor.\ \ref{cor:B} that the right hand map is an isomorphism. But for this proof we only use its injectivity which is very easy to see (cf.\ section \ref{sec:PosThmNewPf}). It therefore suffices to establish that the composed map $\Theta \circ \Phi = \Phi_{C_c^\infty(Z,k)}$ is bijective. We observe that
\begin{align*}
  \End_{\Mod(Z)}(C_c^\infty(Z,k)) & = \Hom_{\Mod(Z)}( \bigcup_{Z' \in \Op_c(Z)} C_c(Z/Z',k), C_c^\infty(Z,k))   \\
                  & = \varprojlim_{Z' \in \Op_c(Z)} \Hom_{\Mod(Z)}(C_c(Z/Z',k), C_c^\infty(Z,k))  \\
                  & = \varprojlim_{Z' \in \Op_c(Z)} \End_{\Mod(Z)}(C_c(Z/Z',k)) \ .
\end{align*}
This shows that $\Phi_{C_c^\infty(Z,k)} = \varprojlim_{Z' \in \Op_c(Z)} \Phi_{C_c(Z/Z',k)}$ and reduces us to the case of a discrete commutative group $Z$. But in this case $C_c(Z,k)$ is a free module of rank one over $k[Z]$.
\end{proof}
\begin{corollary}\label{cor:Z-G}
  The composed map
\begin{equation*}
  \mathfrak{Z}(Z(G)) \xrightarrow[\cong]{\;\ \Phi^{-1}\;} \widehat{k[Z(G)]} \xrightarrow{\ \Phi\ } \mathfrak{Z}(G)
\end{equation*}
between Bernstein centers is an injective homomorphism of $k$-algebras.
\end{corollary}

\section{The universal endomorphism ring}\label{sec:univ-endo-ring}

One of our goals is the computation of the endomorphism ring $\End_{\Mod(G \times G)}(C_c^\infty(G,k))$.

We fix a compact open subgroup $U \subseteq G$. We also choose a set $\cW \subseteq G$ of representatives of the double cosets of $U$ in $G$. For the background for the subsequent points \textbf{A)} to \textbf{D)} see \cite{Vig96} Chap.\ I.5.
\begin{itemize}
  \item[\textbf{A)}] \textbf{Transitivity of induction:} We have the $G_\ell$-equivariant isomorphism
  \begin{align}\label{f:trans-ind}
    \ind_{U_\ell}^{G_\ell}(C^\infty(U,k)) & \xrightarrow{\cong} C_c^\infty(G,k)   \\
                                         \phi & \longmapsto F_\phi(g) := \phi(g)(1) \ .  \nonumber
  \end{align}
  Note that $\phi(g)(u) = F_\phi(gu)$ for $u \in U$. One checks that the $G_r$-action on $C_c^\infty(G,k)$ corresponds under this isomorphism to the $G_r$-action on $ \ind_{U_\ell}^{G_\ell}(C^\infty(U,k))$ given by
  \begin{equation}\label{f:Gr}
    (h,\phi) \longmapsto h_r(\phi)(g)(u) := \phi(guh)(1) \ .
  \end{equation}
  In particular, for $h = v \in U$ we have
\begin{equation*}
  v_r(\phi)(g)(u) = \phi(guv)(1) = {^{(uv)_\ell^{-1}}(}\phi(g))(1) = \phi(g)(uv) = {^{v_r}(}\phi(g))(u)
\end{equation*}
  and hence
\begin{equation}\label{f:Ur}
  v_r(\phi)(g) = {^{v_r}(}\phi(g)) \ .
\end{equation}

  \item[\textbf{B)}] \textbf{Frobenius reciprocity:} We have the $U_\ell$-equivariant embedding
\begin{align*}
  C^\infty(U,k) & \longrightarrow \ind_{U_\ell}^{G_\ell}(C^\infty(U,k))   \\
            f & \longmapsto \iota(f)(g) :=
            \begin{cases}
            {^{g_\ell^{-1}} f} & \text{if $g \in U$},  \\
            0 & \text{otherwise}.
            \end{cases}
\end{align*}
  It gives rise to the isomorphism
\begin{align}\label{f:frob-rec}
  \End_{\Mod(G_\ell)}(\ind_{U_\ell}^{G_\ell}(C^\infty(U,k))) & \xrightarrow{\;\cong\;} \Hom_{\Mod(U_\ell)}(C^\infty(U,k), \ind_{U_\ell}^{G_\ell}(C^\infty(U,k)))  \\
                                                  \alpha & \longmapsto \alpha^\flat := \alpha \circ \iota \ .      \nonumber
\end{align}
  Later on we will be interested only in $G_r$-equivariant $\alpha$'s. The main point being, as it will turn out, the $U_r$-equivariance we note here only the following. For $v \in U$ we compute, using \eqref{f:Ur} in the last equality,
\begin{equation*}
  \iota({^{v_r} f})(g) =
  \begin{cases}
  {^{g_\ell^{-1}}(}{^{v_r}f}) \\
  0
  \end{cases}
  =
  \begin{cases}
  {^{v_r}(}{^{g_\ell^{-1}}f}) \\
  0
  \end{cases}
  = {^{v_r} (} \iota(f)(g)) = v_r(\iota(f))(g) \ ,
\end{equation*}
which means that the map $\iota$ is $U_r$-equivariant. Hence with $\alpha$ also $\alpha^\flat$ is $U_r$-equivariant.

  \item[\textbf{C)}] \textbf{Mackey decomposition:} The disjoint decomposition $G = \dot{\bigcup}_{w \in \cW} UwU$ gives rise to the $U_\ell$-invariant decomposition
\begin{equation*}
  \ind_{U_\ell}^{G_\ell}(C^\infty(U,k)) = \oplus_{w \in \cW} \ind_{U_\ell}^{(UwU)_\ell}(C^\infty(U,k)) \ .
\end{equation*}
  We put $U_w := U \cap w U w^{-1}$. For any $w \in \cW$ the map
\begin{align*}
  \ind_{U_\ell}^{(UwU)_\ell}(C^\infty(U,k)) & \xrightarrow{\;\cong\;} \ind_{(U_w)_\ell}^{U_\ell}(w_* \res_{(U_{w^{-1}})_\ell}^{U_\ell}(C^\infty(U,k)))  \\
                         \phi & \longmapsto \tilde{\phi}(u) := \phi(uw)
\end{align*}
  is a $U_\ell$-equivariant isomorphism.\footnote{This map does depend on the choice of the set $\cW$.} Here the functor $w_*$ sends a $U_{w^{-1}}$-representation $Y$ to the $U_w$-representation $w_* Y := Y$ but on which $U_w$ acts through the homomorphism $u \mapsto w^{-1} u w$. We see that
\begin{align}\label{f:mackey}
  \ind_{U_\ell}^{G_\ell}(C^\infty(U,k)) & \xrightarrow{\;\cong\;} \oplus_{w \in \cW}  \ind_{(U_w)_\ell}^{U_\ell}(w_* C^\infty(U,k))      \\
  \phi & \longmapsto (\phi_w(u) := \phi(uw))_w    \nonumber
\end{align}
  is a $U_\ell$-equivariant isomorphism. Using \eqref{f:Ur} one checks that this map also is $U_r$-equivariant where $U_r$ acts on the target in a way which is induced in each summand by the $U_r$-action on $C^\infty(U,k)$.

  \item[\textbf{D)}] \textbf{Second Frobenius reciprocity:} Let $Y$ be any object in $\Mod(U_w)$. Since $U_w$ is open and hence of finite index in $U$ the $U_w$-equivariant map
\begin{align}\label{f:frob2}
  \ind_{U_w}^U(Y) & \longrightarrow Y \\
                \phi & \longmapsto \phi(1)     \nonumber
\end{align}
  induces the isomorphism
\begin{align*}
 \Hom_{\Mod(U_\ell)}(C^\infty(U,k),\ind_{U_w}^U(Y)) & \xrightarrow{\;\cong\;} \Hom_{\Mod((U_w)_\ell)}(C^\infty(U,k),Y)  \\
                                       \beta & \longmapsto [f \mapsto \beta(f)(1)] \ .
\end{align*}
  For $Y = w_* C^\infty(U,k)$ this becomes the isomorphism
\begin{align}\label{f:frob2iso}
 \Hom_{\Mod(U_\ell)}(C^\infty(U,k),\ind_{(U_w)_\ell}^{U_\ell}(w_* C^\infty(U,k))) & \xrightarrow{\;\cong\;} \Hom_{\Mod((U_w)_\ell)}(C^\infty(U,k),w_* C^\infty(U,k)) \nonumber \\
                                       \beta & \longmapsto [f \mapsto \beta(f)(1)] \ .
\end{align}
  In this case the map \eqref{f:frob2} is visibly $U_r$-equivariant. Hence \eqref{f:frob2iso} respects $U_r$-equivariance. \\
\end{itemize}

The combination of \eqref{f:frob-rec}, \eqref{f:mackey}, and \eqref{f:frob2iso} leads to the following chain of maps:
\begin{equation}\label{f:mapchain}
  \xymatrix{
    \End_{\Mod(G_\ell)}(\ind_{U_\ell}^{G_\ell}(C^\infty(U,k))) \ar[d]^{\cong} \\
    \Hom_{\Mod(U_\ell)}(C^\infty(U,k), \ind_{U_\ell}^{G_\ell}(C^\infty(U,k))) \ar[d]^{\cong} \\
    \Hom_{\Mod(U_\ell)}(C^\infty(U,k), \oplus_{w \in \cW} \ind_{(U_w)_\ell}^{U_\ell}(w_* C^\infty(U,k))) \ar[d]^{\subseteq} \\
    \Hom_{\Mod(U_\ell)}(C^\infty(U,k), \prod_{w \in \cW} \ind_{(U_w)_\ell}^{U_\ell}(w_* C^\infty(U,k))) \ar[d]^{=} \\
    \prod_{w \in \cW} \Hom_{\Mod(U_\ell)}(C^\infty(U,k), \ind_{(U_w)_\ell}^{U_\ell}(w_* C^\infty(U,k))) \ar[d]^{\cong} \\
    \prod_{w \in \cW} \Hom_{\Mod((U_w)_\ell)}(C^\infty(U,k),w_* C^\infty(U,k))  }
\end{equation}
We denote the composite map by $\alpha \longmapsto \Omega(\alpha) = (\Omega_w(\alpha))_{w \in \cW}$. Its explicit description is
\begin{equation}\label{f:Omega}
  \Omega_w(\alpha)(f) = (\alpha \circ \iota(f))(w) \ .
\end{equation}
We have seen already that $\Omega$ respects $U_r$-equivariance. It is clear that the image of $\Omega$ is
\begin{multline}\label{f:imOmega}
  \im(\Omega) =   \\
  \{ (\omega_w)_w : \text{for any $f \in C^\infty(U,k)$ we have $\omega_w(f) = 0$ for all but finitely many $w \in \cW$}\}.
\end{multline}

To reformulate the target of $\Omega$ under the presence of the $U_r$-equivariance condition we first observe that
\begin{multline}\label{f:twisted0}
  \Hom_{\Mod(U_w \times U)}(C^\infty(U,k),w_* C^\infty(U,k)) =  \\
   \{\beta \in \End_{\Mod(U_r)}(C^\infty(U,k)) : \beta({^{u_\ell}(}-)) = {^{(w^{-1}uw)_\ell}(}\beta(-)) \; \text{for any $u \in U_w$}\}.
\end{multline}
The $U_\ell$-action on $C^\infty(U,k)$ extends uniquely to an action of the completed group ring $k[[U_\ell]]$. The $U_r$-action is equivariant for this $k[[U_\ell]]$-action. This gives rise to a ring homomorphism
\begin{equation*}
  k[[U_\ell]] \xrightarrow{\;\cong\;} \End_{\Mod(U_r)}(C^\infty(U,k))
\end{equation*}
which is easily seen to be an isomorphism. Using this identification the equality \eqref{f:twisted0} becomes the equality
\begin{multline}\label{f:twisted}
  \Hom_{\Mod(U_w \times U)}(C^\infty(U,k),w_* C^\infty(U,k)) =  \\
   \{\lambda \in k[[U]] : \lambda u = (w^{-1} u w)\lambda \; \text{for any $u \in U_w$}\} =: k[[U]]_w \ .
\end{multline}
We conclude that $\Omega$ restricts to an injective map
\begin{equation}\label{f:OmegaU}
  \Omega_U : \End_{\Mod(G \times U)}(\ind_{U_\ell}^{G_\ell}(C^\infty(U,k))) \longrightarrow \prod_{w \in \cW} k[[U]]_w
\end{equation}
with the additional property, by \eqref{f:imOmega}, that any element $(\lambda_w)_w$ in the image of $\Omega_U$ satisfies:
\begin{equation}\label{f:imOmegaU}
  \text{For any $f \in C^\infty(U,k)$ we have $\lambda_w(f) = 0$ for all but finitely many $w \in \cW$}.
\end{equation}
In section \ref{sec:twisted-action} we will study the vector spaces $k[[U]]_w$ based upon a formalism developed in the next section.

\section{Fixed points in profinite permutation modules}
\label{sec:FixPtsPMM}
In this section, we extend \cite{Ard} Prop.\ 2.1 to a more general setting.

Let $\Gamma$ be a profinite group and let $Y$ be a profinite set equipped with a continuous action $\act : \Gamma \times Y \to Y$ of $\Gamma$. We begin with some topological preliminaries.

\begin{lemma}\label{OpenStab} Let $P$ be a closed and open subset of $Y$. Then the stabiliser $\Gamma_P$ is open in $\Gamma$.
\end{lemma}
\begin{proof} Note that  $\Op_c(\Gamma)$ is a directed set under reverse inclusion. First, we show that
\[\bigcap\limits_{H \in \Op_c(\Gamma)}  HP = P.\]
Let $x$ lie in the intersection; then for each $H \in \Op_c(\Gamma)$ we can find $y_H \in P$ and $h_H \in H$ such that $x = h_H y_H$. Regarding $H \mapsto h_H$ as a net in $\Gamma$, we observe that $h_H \to 1$ because $\bigcap_{H \in \Op_c(\Gamma)} H = \{1\}$. Now $P$ is compact because $P$ is closed in $Y$ and $Y$ is compact. So, by \cite{Kel} Thm.\ 5.2, we can find a subnet $\varphi : B \to \Op_c(\Gamma)$ such that $y_{\varphi(b)} \to y$ for some $y \in P$. Note that $h_{\varphi(b)} \to 1$ still, so $(h_{\varphi(b)}, y_{\varphi(b)}) \to (1,y)$ in $\Gamma \times Y$. Since $\act : \Gamma \times Y \to Y$ is continuous, we conclude that $x = h_{\varphi(b)} \cdot y_{\varphi(b)} \to 1 \cdot y = y$. Therefore $x = y \in P$ because $Y$ is Hausdorff.

Now each $HP$ is the continuous image in $Y$ of the compact space $H \times P$ under $\act$, and is therefore compact. Since $Y$ is Hausdorff, each $HP$ is closed in $Y$. By the above,
\[ Y \backslash P = \bigcup\limits_{H \in \Op_c(\Gamma)} Y \backslash HP\]
is an open covering of $Y \backslash P$. Since $P$ is also open, its complement $Y \backslash P$ is closed and hence compact.  This implies that $P = H_1 P \cap \cdots \cap H_nP$ for some $H_1,\cdots, H_n \in \Op_c(\Gamma)$. But then the open subgroup $H_1 \cap \cdots \cap H_n$ of $\Gamma$ stabilises $P$.
\end{proof}

Now we can record a $\Gamma$-equivariant version of the well-known topological characterisation of profinite sets. Let $\mathfrak{P}$ be the set of open partitions $\cP$ of $Y$ so that every $\cP \in \mathfrak{P}$ is a finite discrete quotient space of $Y$. These form an inverse system if we order $\mathfrak{P}$ by refinement, so that $\cP_2 \geq \cP_1$ if and only if every $P_2 \in \cP_2$ is contained in some $P_1 \in \cP_1$. By \cite{Wil} Prop.\ 1.1.7, the natural map $Y \stackrel{\cong}{\longrightarrow} \varprojlim_{\cP \in \mathfrak{P}}\,\cP$ is an isomorphism of profinite sets.

We say that $\cQ \in \mathfrak{P}$ is \emph{$\Gamma$-stable} if $gQ \in \cQ$ for all $Q \in \cQ$ and all $g \in \Gamma$. Like $\mathfrak{P}$, the set $\mathfrak{Q} \subset \mathfrak{P}$ of $\Gamma$-stable open partitions of $Y$ forms  a directed set under refinement.

\begin{lemma}\label{lem:EqRefsExist}
Every $\cP \in \mathfrak{P}$ admits a $\Gamma$-stable refinement $\cQ \in \mathfrak{Q}$.
\end{lemma}
\begin{proof}
Because $Y$ is compact, $\cP$ is a finite set, so using Lemma \ref{OpenStab} we can find an open subgroup $H$ of $\Gamma$ which stabilises each member of $\cP$. Fix a complete set $\{g_1,\cdots,g_k\}$ of left coset representatives of $H$ in $\Gamma$, and define
\[\cQ := \{ g_1P_1 \cap \cdots \cap g_k P_k : P_1,\cdots, P_k \in \cP \} \backslash \{\emptyset\}.\]
Then every member of $\cQ$ is a non-empty closed and open subset of $\Gamma$, and since $Y = \bigcup_{P \in \cP} g_i P$ for each $i = 1, \cdots, k$, we see that
\[Y = \bigcap_{i=1}^k \bigcup\limits_{P \in \cP} g_i P = \bigcup\limits_{P_1,\cdots,P_k \in \cP}g_1 P_1 \cap \cdots \cap g_k P_k = \bigcup\limits_{Q \in \cQ} Q.\]
So, $\cQ$ forms a covering of $Y$ by non-empty closed and open subsets. Let $Q = g_1 P_1 \cap \cdots \cap g_k P_k$ and $Q' = g_1 P_1' \cap \cdots \cap g_k P_k'$ be two members of $\cQ$ such that $Q \cap Q' \neq \emptyset$. Then
\[ \emptyset \neq Q \cap Q' = \left( g_1 P_1 \cap \cdots \cap g_k P_k \right) \cap \left( g_1 P'_1 \cap \cdots \cap g_k P'_k \right) = g_1 (P_1 \cap P_1') \cap \cdots \cap g_k (P_k \cap P_k')\]
shows that $P_i \cap P_i' \neq \emptyset$ for each $i=1,\cdots, k$. Since $\cP$ is a partition, $P_i = P_i'$ for each $i=1,\cdots, k$ and hence $Q = Q'$. So, $\cQ$ is a partition of $Y$.

Finally, to see that $\cQ$ is $\Gamma$-stable, consider the permutation action of $\Gamma$ on $\Gamma / H$. Fix $x \in \Gamma$ and write $x g_iH = g_{x \cdot i} H$ for some permutation $i \mapsto x \cdot i$ of $\{1,\cdots, k\}$. Since $H$ stabilises each $P \in \cP$, we have $x g_i P = g_{x \cdot i} P$ for all $i = 1,\cdots,k$ and all $P \in \cP$. Then for any $Q = gP_1 \cap \cdots \cap g_k P_k \in \cQ$, we see that
\[ xQ = x(g_1P_1\cap \cdots \cap g_k P_k) = g_{x\cdot 1} P_1 \cap \cdots \cap g_{x \cdot k} P_k = g_1 P_{x^{-1} \cdot 1} \cap \cdots \cap g_k P_{x^{-1} \cdot k} \in \cQ.\]
Thus $\cQ$ is the required $\Gamma$-stable refinement of $\cP$.
\end{proof}

We regard each $\cQ \in \mathfrak{Q}$ as a finite discrete space; it admits a natural $\Gamma$-action. Lemma \ref{OpenStab} implies that this action is continuous. If $\cQ \in \mathfrak{Q}$ is a refinement of $\cR \in \mathfrak{Q}$, then the natural map $\pi_{\cQ\cR} : \cQ \to \cR$ which sends $Q \in \cQ$ to the unique element $R \in \cR$ containing $Q$ is $\Gamma$-equivariant, so the profinite set $\varprojlim_{\cQ \in \mathfrak{Q}} \cQ$ carries a natural continuous $\Gamma$-action. From Lemma \ref{lem:EqRefsExist} we easily deduce the following

\begin{corollary}\label{lem:EqPtsCofinal}
The natural $\Gamma$-equivariant map $Y \xrightarrow{\;\cong\;} \varprojlim_{\cQ \in \mathfrak{Q}}\cQ$ is an isomorphism.
\end{corollary}
Let $R$ be any possibly noncommutative ring.

\begin{definition}
The \emph{permutation module} $R[[Y]]$ of $Y$ is the projective limit
\[ R[[Y]] := \varprojlim_{\cQ \in \mathfrak{Q}} \, R[ \cQ ]\]
of the usual permutation $R[\Gamma]$-modules $R[ \cQ ]$ as $\cQ$ ranges over the set $\mathfrak{Q}$ of $\Gamma$-stable open partitions of $Y$. It is equipped with the projective limit topology of the discrete topologies.
\end{definition}

By functoriality, there is a natural $\Gamma$-action on $R[[Y]]$ and we would like to describe the $R$-submodule of $\Gamma$-invariants $R[[Y]]^\Gamma$ in $R[[Y]]$.

\begin{definition} \,
\begin{enumerate}
\item Let $\Gamma \backslash Y$ denote the set of $\Gamma$-orbits in $Y$.
\item Let $\Delta_\Gamma(Y) := \{ y \in Y : |\Gamma \cdot y| < \infty\}$ be the set of points whose $\Gamma$-orbit is finite.
\item For each $\cO \in \Delta_\Gamma(Y)$, let $\upsilon_Y(\cO) := \sum\limits_{y \in \cO} y \in R[Y] \subset R[[Y]]$ denote the \emph{orbit sum} of $\cO$.
\end{enumerate}
\end{definition}

Any such finite orbit sum is $\Gamma$-invariant, so we obtain a natural $R$-linear map
\[ \upsilon_Y : R[\Gamma \backslash \Delta_\Gamma(Y)] \to R[[Y]]^\Gamma\]
which sends the basis vector $\cO \in R[\Gamma \backslash \Delta_\Gamma(Y)]$ to its orbit sum $\upsilon_Y(\cO)$. Evidently this map is injective.

\begin{remark}
Since the action of $\Gamma$ on $Y$ is continuous, the stabiliser $\Gamma_y$ of any point $y \in Y$ is closed. Since a closed subgroup of finite index in $\Gamma$ is open, we see that $y \in Y$ lies in $\Delta_\Gamma(Y)$ if and only if $y$ is fixed by some open subgroup $H$ of $\Gamma$, which we may as well take to be normal:
\[ \Delta_\Gamma(Y) = \bigcup\limits_{H \in \Op_c(\Gamma)} Y^H = \bigcup\limits_{N \in \Op_c(\Gamma)\,\text{normal}} Y^N.  \]
\end{remark}

We begin with the well-known, but instructive case where $Y$ is finite.

\begin{lemma}\label{FiniteInvariants}
Suppose that $Y$ is finite. Then $\upsilon_Y$ is a bijection.
\end{lemma}

For any two $\cQ \geq \cR$ in $\mathfrak{Q}$ we introduce the $R$-linear map
\begin{align*}
  \sigma_{\cQ\cR} : R[\Gamma \backslash \cQ] & \longrightarrow R[\Gamma \backslash \cR] \\
              \cC & \longmapsto \frac{|\cC|}{|\pi_{\cQ\cR}(\cC)|} \pi_{\cQ\cR}(\cC) \ .
\end{align*}
Note that the fraction actually is an integer. It is easy the check that the diagram
\begin{equation}\label{f:sigmaQR}
  \xymatrix{ R[ \Gamma \backslash \cQ] \ar[rr]^{\sigma_{\cQ\cR}}\ar[d]^{\cong}_{\upsilon_{\cQ}}    &&  R[ \Gamma \backslash \cR] \ar[d]_{\cong}^{\upsilon_{\cR}}\\
R[\cQ]^\Gamma \ar[rr]_{R[\pi_{\cQ\cR}]} && R[\cR]^\Gamma }
\end{equation}
is commutative. By passing to the projective limit we obtain an $R$-linear topological isomorphism
\begin{equation}\label{f:GenInv}
  \varprojlim_{\cQ \in \mathfrak{Q}} R[ \Gamma \backslash \cQ] \xrightarrow{\;\cong\;} R[[Y]]^\Gamma
\end{equation}
w.r.t.\ the projective limit topologies. Note that on the target this topology coincides with the subspace topology from $R[[Y]]$.

We now specialise to the case where $R$ is a ring of characteristic $p$ and $\Gamma$ is a pro-$p$ group. In this case, we have the following description.

\begin{lemma}\label{lem:currents}
Let $x = (x_{\cQ})_{\cQ \in \mathfrak{Q}} \in \prod\limits_{\cQ \in \mathfrak{Q}} R[\Gamma \backslash \cQ]$ and write $x_{\cQ} = \sum\limits_{\cC \in \Gamma \backslash \cQ} x(\cC) \, \cC$ for all $\cQ \in \mathfrak{Q}$. Then $x \in  \varprojlim_{\cQ \in \mathfrak{Q}}  R[\Gamma \backslash \cQ]$ if and only if for all $\cP \geq \cQ$ in $\mathfrak{Q}$ and all $\cC \in \Gamma \backslash \cQ$ we have
\[ x(\cC) = \sum_{\cB} x(\cB)\]
where the sum runs over the finite set $\{\cB \in \Gamma \backslash \cP : \pi_{\cP\cQ}(\cB) = \cC$ and $|\cB| = |\cC|\}$.
\end{lemma}
\begin{proof} Let $\cP \geq \cQ$ and write $x_{\cP} = \sum\limits_{\cB \in \Gamma \backslash \cP} x(\cB) \cB$. Then by \eqref{f:sigmaQR}, we have
\[\sigma_{\cP\cQ}(x_{\cP}) = \sum\limits_{\cC \in \Gamma \backslash \cQ} \left(\sum\limits_{\pi_{\cP\cQ}(\cB)=\cC}  x(\cB) \frac{ |\cB| }{| \cC| }\right) \cC. \]
Let $\cB \in \Gamma \backslash \cP$ and let $\cC = \pi_{\cP\cQ}(\cB)$. Then $|\cB| / |\cC| = [\Gamma_{\pi_{\cP\cQ}(P)} : \Gamma_P]$ for any $P \in \cB$. Since $\Gamma$ is a pro-$p$ group, the index $[\Gamma_{\pi_{\cP\cQ}(P)}:\Gamma_P]$ is a power of $p$. Since $R$ is a ring of characteristic $p$, $|\cB|/|\cC|$ is zero in $R$ if $|\cB| > |\cC|$. Hence
\[\sigma_{\cP\cQ}(x_{\cP}) = \sum\limits_{\cC \in \Gamma \backslash \cQ}\left(\sum_{\substack{\pi_{\cP\cQ}(\cB)=\cC  \\  |\cB| = |\cC|}}  x(\cB) \right) \cC.\]
The result follows, because $x \in \varprojlim_{\cQ \in \mathfrak{Q}}  R[\Gamma \backslash \cQ]$ if and only if $\sigma_{\cP\cQ}(x_{\cP}) = x_{\cQ}$ for all $\cP \geq \cQ$.
\end{proof}

Here is our generalisation of \cite{Ard} Prop.\ 2.1.

\begin{proposition}
\label{FixedPoint} Suppose that $R$ is a ring of characteristic $p$ and that $\Gamma$ is a pro-$p$ group. Then the map $\upsilon_Y : R[\Gamma \backslash \Delta_\Gamma(Y)] \to R[[Y]]^\Gamma$ has dense image.
\end{proposition}
\begin{proof}
In view of \eqref{f:GenInv}, it is enough to show that the natural $R$-linear map
\[\eta : R[\Gamma \backslash \Delta_\Gamma(Y)] \to \varprojlim_{\cQ \in \mathfrak{Q}} R[ \Gamma \backslash \cQ]\]
has dense image. Let $\sigma_{\cR} : \varprojlim_{\cQ \in \mathfrak{Q}} R[ \Gamma \backslash \cQ]  \to R[\Gamma \backslash \cR]$ be the canonical projection maps; by definition of the projective limit topology, it will be enough to show that
\begin{equation}\label{f:density}
\sigma_{\cR}(\im \eta) = \im \sigma_{\cR} \quad\text{for all $\cR \in \mathfrak{Q}$}.
\end{equation}
Fix $\cR \in \mathfrak{Q}$ and $\cD \in \Gamma \backslash \cR$. For each $\cQ \geq \cR$ in $\mathfrak{Q}$, define
\[ \cS(\cQ) := \{ \cC \in \Gamma \backslash \cQ : \pi_{\cQ\cR}(\cC) = \cD \ \mbox{and}\ |\cC| = |\cD| \}.\]
Suppose $\cP \geq \cQ \geq \cR$ and pick $\cB \in \cS(\cP)$. Let $\cC := \pi_{\cP\cQ}(\cB) \in \Gamma\backslash \cQ$; then $\pi_{\cQ\cR}(\cC) = \pi_{\cQ\cR}\pi_{\cP\cQ}(\cB) = \pi_{\cP\cR}(\cB) = \cD$ and hence $|\cB| \geq |\pi_{\cP\cQ}(\cB)| = |\cC| \geq |\pi_{\cQ\cR}(\cC)| = |\cD| = |\cB|$ which implies that $|\cC| = |\cD|$. Hence $\cC \in \cS(\cQ)$. In view of \eqref{f:sigmaQR}, we see that the transition map $\sigma_{\cP\cQ} : R[\Gamma\backslash \cP] \to R[\Gamma\backslash \cQ]$ restricts to a well-defined map $\sigma_{\cP\cQ} : \cS(\cP) \to \cS(\cQ)$, so the finite sets $(\cS(\cQ))_{\cQ \geq \cR}$ form a projective system.

Now, suppose that there exists $x \in \varprojlim_{\cQ \in \mathfrak{Q}} R[ \Gamma \backslash \cQ]$ such that $x(\cD) \neq 0$ (notation as in Lemma \ref{lem:currents}). For every $\cQ \geq \cR$, we deduce from Lemma \ref{lem:currents} that $x(\cC) \neq 0$ for at least one $\cC \in \Gamma \backslash \cQ$ with $\pi_{\cQ\cR}(\cC) = \cD$ and $|\cC| = |\cD|$. Hence $\cS(\cQ)$ is non-empty for every $\cQ \geq \cR$. Since the indexing set $\{\cQ \in \mathfrak{Q} : \cQ \geq \cR\}$ of this projective system is directed, we deduce from \cite{Bou} Chap. I, $\S 9.6$, Prop. 8(b) that the projective limit $\varprojlim_{\cQ \geq \cR} \cS(\cQ)$ is non-empty. Choose a family $(\cC(\cQ))_{\cQ \geq \cR}$ in this limit. Then since each member of this family is a finite set of size $|\cD|$, it follows from the definition of $\cS(\cQ)$ that the transition maps $\pi_{\cP\cQ} : \cP \to \cQ$ restrict to bijections $\pi_{\cP\cQ} : \cC(\cP) \to \cC(\cQ)$ for every $\cP\geq \cQ \geq \cR$. Therefore the projective limit $\varprojlim_{\cQ \geq \cR} \cC(\cQ)$ is non-empty. So we may choose an element $(P(\cQ))_{\cQ \geq \cR}$ in this projective limit. Although we think of $P(\cQ)$ as a point in the finite set $\cQ$ we may view $P(\cQ)$ correspondingly as a non-empty open and closed subset of $Y$. By definition of $\pi_{\cQ\cR}$, we see that for each $\cP \geq \cQ \geq \cR$, $P(\cP)$ is contained in $P(\cQ)$. Since $\mathfrak{Q}$ is directed and forms a cofinal family of finite open partitions of $Y$ by Lemma \ref{lem:EqRefsExist} and since $Y$ is a profinite set, we see that $\bigcap_{\cQ \geq \cR} P(\cQ)$ is a singleton point, $\{y_{\cD}\}$ say. Now, $\Gamma_{P(\cP)} \subseteq \Gamma_{P(\cQ)}$ for all $\cP \geq \cQ \geq \cR$ since $\pi_{\cP\cQ} : \cP \to \cQ$ is $\Gamma$-equivariant and since its restriction $\pi_{\cP\cQ} : \cC(\cP) \to \cC(\cQ)$ is a bijection mapping $P(\cQ)$ to $P(\cP)$. On the other hand, $[\Gamma : \Gamma_{P(\cQ)}] = |\Gamma \cdot P(\cQ)| = |\cC(\cQ)| = |\cD|$ shows that in fact $\Gamma_{P(\cQ)} = \Gamma_{P(\cR)}$ for all $\cQ \geq \cR$. Since $\bigcap_{\cQ \geq \cR} \Gamma_{P(\cQ)}$ evidently stabilises $\bigcap_{\cQ \geq \cR} P(\cQ)$, we conclude that the finite index subgroup $\Gamma_{P(\cR)}$  of $\Gamma$ fixes $y_{\cD}$ and hence $y_{\cD} \in \Delta_\Gamma(Y)$. Hence $\sigma_{\cR}(\eta(\Gamma \cdot y_{\cD})) = \Gamma \cdot P(\cR) = \cD$.

Finally, since the inclusion $\sigma_{\cR}(\im \eta) \subseteq \im \sigma_{\cR}$ in \eqref{f:density} is clear, take $x = (x_{\cQ})_{\cQ \in \mathfrak{Q}} \in \varprojlim_{\cQ \in \mathfrak{Q}} R[ \Gamma \backslash \cQ]$ and consider its image $\sigma_{\cR}(x) = x_{\cR} = \sum\limits_{x(\cD) \neq 0} x(\cD) \cD \in R[\Gamma \backslash \cR]$. Then
\[x_\cR = \sigma_{\cR} \eta \left( \sum\limits_{x(\cD) \neq 0} x(\cD) \, \Gamma \cdot y_{\cD} \right) \in \sigma_{\cR}(\im \eta).\]
This shows that $\im \sigma_{\cR} \subseteq \sigma_{\cR}(\im \eta)$ and completes the proof.
\end{proof}

\begin{corollary}\label{cor:OrbitSumsGen}
Let $\Gamma$ be a pro-$p$ group which acts continuously on the profinite set $Y$ and let $R$ be a ring of characteristic $p$. Suppose that $\Delta_\Gamma(Y)$ is closed in $Y$. Then $\Delta_\Gamma(Y)$ is also a profinite set, and the natural maps
\[j : R[[\Gamma \backslash \Delta_\Gamma(Y)]] \to R[[\Delta_\Gamma(Y)]]^\Gamma \quad \mbox{and} \quad i : R[[\Delta_\Gamma(Y)]]^\Gamma \to R[[Y]]^\Gamma\]
are both isomorphisms.
\end{corollary}
\begin{proof} The first assertion is clear since any closed subspace of a profinite set is again profinite. The two maps in question appear in the following commutative diagram
\[\xymatrix{ R[\Gamma \backslash \Delta_\Gamma(Y)] \ar[rr] \ar[d]_{\upsilon_Y} \ar[rrd]_{\upsilon_{\Delta_\Gamma(Y)}} && R[[\Gamma \backslash \Delta_\Gamma(Y)]] \ar[d]^j \\
R[[Y]]^\Gamma  && R[[\Delta_\Gamma(Y)]]^\Gamma\ar[ll]^{i}
}\]
and are both injective. The maps $\upsilon_Y$ and $\upsilon_{\Delta_\Gamma(Y)}$ both have dense image by Prop. \ref{FixedPoint}, since $\Delta_\Gamma(Y)$ is profinite and since $\Delta_\Gamma(\Delta_\Gamma(Y)) = \Delta_\Gamma(Y)$. The maps $i$ and $j$ have closed image, by the functoriality of the $Y \mapsto R[[Y]]$ construction. Hence they are both isomorphisms.
\end{proof}

In general, $\Delta_\Gamma(Y)$ need not be closed in $Y$ as the following example shows.

\begin{example} Let $\Gamma = \mathbb{Z}_p$. For each $n \geq 0$, let $Y_n$ be the disjoint union of $n+1$ $\Gamma$-orbits $Y_n(0), Y_n(1), \cdots, Y_n(n)$, where $|Y_n(m)| = p^m$ for all $n \geq m \geq 0$. Choose base points $y_n(m) \in Y_n(m)$ for each $n \geq m \geq 0$. For each $n \geq 0$, define $\pi_{n+1} : Y_{n+1} \to Y_n$ to be the unique $\Gamma$-equivariant map such that
\[ \pi_{n+1}(y_{n+1}(m)) = \left\{\begin{array}{lll}  y_n(m) & \quad \mbox{if} \quad & n \geq m \geq 0 \\ y_n(n) & \quad \mbox{if} \quad & m = n+1 .\end{array} \right.\]
The restriction of $\pi_{n+1}$ to $Y_{n+1}(m)$ is an isomorphism onto $Y_n(m)$ whenever $n \geq m \geq 0$. The projective limit $Y = \varprojlim Y_n$ is naturally a profinite set equipped with a continuous $\Gamma$-action.
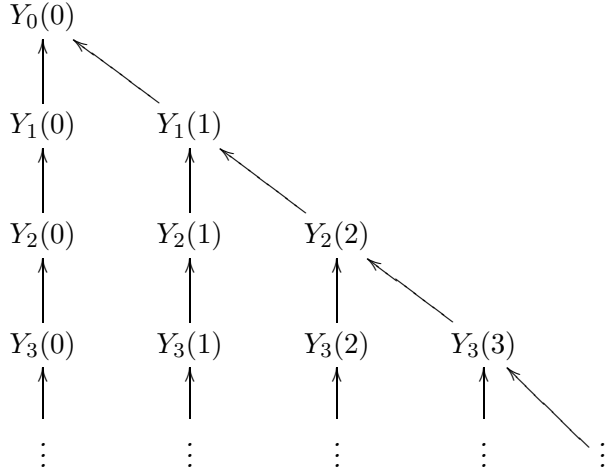
\[\xymatrix{ Y_0(0) \\ Y_1(0) \ar[u] & Y_1(1) \ar[ul] \\ Y_2(0) \ar[u] & Y_2(1) \ar[u]& Y_2(2) \ar[ul]\\  Y_3(0) \ar[u]& Y_3(1) \ar[u]& Y_3(2)  \ar[u]& Y_3(3) \ar[ul] \\ \vdots \ar[u] & \vdots \ar[u] & \vdots \ar[u] & \vdots \ar[u]& \vdots \ar[ul] }\]
For each $m \geq 0$, let $y_\infty(m) := (y_n(m))_{n = 0}^\infty \in Y$. Then the $\Gamma$-orbit of $y_\infty(m)$ has size exactly $p^m$ for all $m \geq 0$, so that $y_\infty(m) \in \Delta_\Gamma(Y)$ for all $m \geq 0$.

Suppose now that $z \in Y$ has infinite $\Gamma$-orbit. Write $z = (z_n)_{n=0}^\infty$ where $z_n \in Y_n$ and suppose for a contradiction that for some $n \geq 0$, $z_n \notin Y_n(n+1)$. Then $z_n \in Y_n(m)$ for some $m \leq n$. Then for all $k \geq n$, $z_k$ must lie in the preimage of $Y_n(m)$ in $Y_k$, which is $Y_k(m)$ by construction. But then $p^m \mathbb{Z}_p$ fixes $z_k$ for all $k \geq n$ and hence $p^m \mathbb{Z}_p$ fixes $z$ as well. This contradicts the hypothesis that the $\Gamma$-orbit of $z$ was infinite. Therefore $z_n$ must lie in $Y_n(n+1)$ for each $n \geq 0$, and we conclude that $Y \backslash \Delta_\Gamma(Y)$ consists of a single infinite $\Gamma$-orbit. Since this orbit is necessarily closed, it follows that $\Delta_\Gamma(Y) = \bigcup\limits_{m = 0}^\infty \Gamma \cdot y_\infty(m)$ is open.

Finally, for any $n \geq 0$, the image of both $z$ and $y_\infty(n+1)$ in $Y_n$ is contained in $Y_n(n+1)$. Hence the image of $z$ in $Y_n$ is contained in the image of $\Delta_\Gamma(Y)$. Hence $z$ lies in the closure of $\Delta_\Gamma(Y)$, but $z \notin \Delta_\Gamma(Y)$. So, $\Delta_\Gamma(Y)$ is not closed.
\end{example}

\section{Twisted actions on completed group rings}\label{sec:twisted-action}

From here on we assume that the field $k$ has \textbf{characteristic $p$.} We keep our fixed compact open subgroup $U$ of $G$ but assume that it is pro-$p$, and we let $w \in G$ be any fixed element. Recall that $U_w = U \cap wUw^{-1}$ is an open subgroup of $U$.
We introduce the twisted conjugation action of $U_w$ on $U$, given by
\begin{align*}
   U_w \times U & \longrightarrow U  \\
          (u,x) & \longmapsto (w^{-1} u w) x u^{-1} \ .
\end{align*}
Let $\Delta_{U_w}(U)$ be the set of finite $U_w$-orbits in $U$ under this action and for any such orbit $\cO$ recall that $\upsilon_U(\cO) \in k[U]$ denotes the sum of all its elements.
This action extends to a twisted conjugation action of $U_w$ on $k[[U]]$, and $k[[U]]_w = k[[U]]^{U_w, *}$ simply is the subspace of fixed points under this action. Since $U$ is pro-$p$, Prop.\ \ref{FixedPoint} gives the following result.

\begin{proposition}\label{prop:Ard}
   The $k$-vector subspace of $k[U]$ spanned by the orbit sums $\upsilon_U(\cO)$ for $\cO \in \Delta_{U_w}(U)$ is dense in $k[[U]]^{U_w, *} = k[[U]]_w$.
\end{proposition}

Next we make the following simple observation.

\begin{lemma}\label{lem:centraliser}
  The $U_w$-orbit of an element $x \in U$ under the twisted action is finite if and only if the centraliser $C_G(wx)$ of the element $wx \in G$ is open in $G$.
\end{lemma}
\begin{proof}
Obviously the orbit of $x$ is finite if and only if the stabiliser $S_x := \{u \in U_w : xu = (w^{-1} u w)x \}$ is open in $U_w$. But for a general $g \in G$ (with $w,x$ fixed) we have $xg = (w^{-1} g w)x$ if and only if $(wx)g = g(wx)$ if and only if $g \in C_G(wx)$. So, $S_x = C_G(wx) \cap U_w$. Thus, the orbit of $x$ is finite if and only if $S_x$ is open in $U_w$ if and only if $C_G(wx)$ is open in $G$.
\end{proof}

\begin{lemma}\label{rem:centers}
  Suppose that $G$ contains no proper open centralisers. Then $Z(U) = Z(G) \cap U$.
\end{lemma}
\begin{proof}
Obviously $Z(G) \cap U \subseteq Z(U)$. Let therefore $x \in Z(U)$. Then $U \subseteq C_G(x)$ so that $C_G(x)$ is open in $G$. Since $G$ contains no proper open centralisers, we conclude that $x \in Z(G)$.
\end{proof}

We assume from now on that $G$ contains no proper open centralisers.

\begin{proposition}\label{prop:twistedDelta}
For any $w \in \cW$, we have $\Delta_{U_w}(U) = U \cap w^{-1} Z(G)$.
\end{proposition}
\begin{proof} The $U_w$-orbit of $x \in U$ is finite if and only if the centraliser $C_G(wx)$ is open in $G$ by Lemma \ref{lem:centraliser}. But this is equivalent to $wx \in Z(G)$ because $G$ contains no proper open centralisers.
\end{proof}

\begin{corollary}\label{cor:CalcTwDelta}
Let $w \in \cW$. Then
\begin{equation*}
  \Delta_{U_w}(U) =
  \begin{cases}
     U \cap Z(G) & \text{if $w \in Z(G)$},  \\
     \emptyset & \text{if $w \not\in Z(G)U$}.
  \end{cases}
\end{equation*}
\end{corollary}
\begin{proof}
This follows immediately from Prop. \ref{prop:twistedDelta}.
\end{proof}
For convenience we choose a set of representatives $\cZ \subseteq Z(G)$ for the double cosets of $U$ contained in $Z(G)U$. We further assume that $\cZ \subseteq \cW$.

\begin{proposition}\label{prop-Ard2}
For any $w \in \cW$, we have
\begin{equation*}
  k[[U]]_w =
  \begin{cases}
  k[[Z(U)]] & \text{if $w \in \cZ$}, \\
  0   & \text{if $w \notin \cZ$}.
  \end{cases}
\end{equation*}
\end{proposition}
\begin{proof}
Prop. \ref{prop:twistedDelta} implies that $\Delta_{U_w}(U)$ is closed in $U$, so we may apply Cor. \ref{cor:OrbitSumsGen} to obtain
\[ k[[U]]_w = k[[U]]^{U_w, \ast} = k[[\Delta_{U_w}(U)]]^{U_w,\ast}.\]
Since $\Delta_{U_w}(U) = \emptyset$ when $w \notin \cZ$ by Cor. \ref{cor:CalcTwDelta}, we may assume that $w \in \cZ$. But then since $w \in Z(G)$, we have $U_w = U$ and the twisted action is the usual conjugation action of $U$ on itself. The result follows, because $\Delta_{U_w}(U) = Z(U)$ by Cor. \ref{cor:CalcTwDelta} and Lemma \ref{rem:centers}.
\end{proof}

Write $Z := Z(G)$. The following is straightforward.

\begin{lemma}\label{lem:ResSurj} The restriction map $\Res^G_{Z} : C^\infty_c(G,k) \to C^\infty_c(Z,k)$ is surjective, and $\widehat{k[Z]}$-equivariant.
\end{lemma}

Let $C := \ind_{U_\ell}^{G_\ell} (C^\infty(U,k))$ and let $C_Z := \ind_{Z(U)_\ell}^{Z_\ell} (C^\infty(Z(U),k))$.  From \eqref{f:OmegaU} together with Prop. \ref{prop-Ard2} we have the injective map
\begin{equation*}
  \Omega_U : \End_{\Mod(G \times U)} (C) \longrightarrow \prod\limits_{z \in \cZ} \End_{\Mod(Z(U))} (C^\infty(Z(U),k)) = \prod\limits_{z \in \cZ} k[[Z(U)]] \ .
\end{equation*}
On the other hand, since $\cZ$ is a transversal for $Z(U)$ in $Z$, replacing $G$ by $Z$ gives us the map $\Omega_{Z(U)}$ out of $\End_{\Mod(Z)} (C_Z)$ with the same target. These two maps appear in the following diagram:
\begin{equation}\label{f:GtoZ}
\xymatrix{ \widehat{k[Z]} \ar[rr]^-{\Phi_C} \ar[d]_{
\Phi_{C_Z}}
 && \End_{\Mod(G \times G)}(C) \ar[r]^{\subseteq} & \End_{\Mod(G \times U)}(C) \ar[d]^{\Omega_U} \\ \End_{\Mod(Z)} (
 C_Z) \ar[rrr]_-{\Omega_{Z(U)}} &&& \prod\limits_{z \in \cZ} k[[Z(U)]]}
\end{equation}

\begin{lemma}\label{lem:DiagComm}
  The diagram \eqref{f:GtoZ} is commutative.
\end{lemma}
\begin{proof}
First, consider the following diagram:
\begin{equation}\label{f:ResFI}
\xymatrix{ C^\infty(U,k) \ar[rr]^-{\iota^G}\ar[d]_{\Res^U_{Z(U)}} && \ind_{U_\ell}^{G_\ell} (C^\infty(U,k)) \ar[rr]^-{F_G}_-{\cong} \ar@{..>}[d]^\psi && C^\infty_c(G,k) \ar[d]^{\Res^G_Z} \\
C^\infty(Z(U),k) \ar[rr]_-{\iota^Z} && \ind_{Z(U)}^Z (C^\infty(Z(U),k)) \ar[rr]_-{F_Z}^-{\cong} && C^\infty_c(Z,k)
}\end{equation}
where $F_G$ and $F_Z$ are the maps of Transitivity of Induction from \eqref{f:trans-ind}, and $\iota^G$ and $\iota^Z$ are the Frobenius reciprocity embeddings appearing in \eqref{f:frob-rec}. From the definitions of $F$ and $\iota$ it is straightforward to verify that the composed horizontal maps are simply the extension by zero maps. As $Z(U) = U \cap Z$ by Lemma \ref{rem:centers} the outer rectangle is commutative. Since $F_G$ and $F_Z$ are isomorphisms in $\Mod(G_\ell)$ and since $\Res^G_Z : C^\infty_c(G,k) \to C^\infty_c(Z,k)$ is $\widehat{k[Z]}$-linear by Lemma \ref{lem:ResSurj}, there is a unique $\widehat{k[Z]}$-linear map $\psi$ making the entire diagram commutative. Note that the inverse of $F_G$ is given as follows: $F_G^{-1}(f)(g)(u) = f(gu)$ for all $f \in C^\infty_c(G,k), g \in G$ and $u \in U$. Now, given $q \in \ind_{U_\ell}^{G_\ell} C^\infty(U,k)$, $z \in Z$ and $y \in Z(U)$, we calculate
\begin{align*}
\psi(q)(z)(y) &= (F_Z^{-1} \circ \Res^G_Z \circ F_G)(q) (z)(y) = (\Res^G_Z F_G(q))(zy) \\
&= F_G(q)(zy) = q(zy)(1)= ({}^{y_\ell^{-1}}q(z))(1) = q(z)(y).
\end{align*}
Therefore the map $\psi$ is given on $q \in \ind_{U_\ell}^{G_\ell} (C^\infty(U,k))$ by the explicit formula
\begin{equation}\label{f:ResRes}
\psi(q) = \Res^U_{Z(U)} \circ \Res^G_Z q.
\end{equation}
Fix $x \in \widehat{k[Z]}$. Let $\alpha := \Phi_C(x)$ denote the action of $x$ on $C$, and let $\beta := \Phi_{C_Z}(x)$ denote the action of $x$ on $C_Z$. Using formula \eqref{f:Omega}, we see that to prove the Lemma, it will be enough to show that $\Omega_z(\alpha) = \Omega_z(\beta)$ for every $z \in \cZ$. To this end, fix $h \in C^\infty(Z(U),k)$ and choose $f \in C^\infty(U,k)$ such that $h = \Res^U_{Z(U)} f$ using Lemma \ref{lem:ResSurj} applied to $U$ in place of $G$. Then we can calculate as follows:
\begin{align*}
 \Omega_z(\beta)(h) & = \beta(\iota^Z(\Res^U_{Z(U)} f))(z)  \qquad\text{by \eqref{f:Omega}} \\
& = \beta(\psi (\iota^G f))(z)  \quad\qquad\qquad\text{by \eqref{f:ResFI}} \\
& = \psi(\alpha(\iota^G f))(z) \\
& = \Res^U_{Z(U)} (\alpha(\iota^G f)(z))  \,\qquad\text{by \eqref{f:ResRes}} \\
& = \Res^U_{Z(U)}  \Omega_z(\alpha)(f)    \quad \qquad\text{by \eqref{f:Omega}}  \\
& = \Omega_z(\alpha)(\Res^U_{Z(U)}f)\\
& = \Omega_z(\alpha)(h)
\end{align*}
where on the third line we used the $\widehat{k[Z]}$-linearity of $\psi$, and on the sixth line we used the $k[[Z(U)]]$-linearity of $\Res^U_{Z(U)} : C^\infty(U,k) \to C^\infty(Z(U),k)$. The result follows, since $k[[Z(U)]]$ acts faithfully on $C^\infty(Z(U),k)$.
\end{proof}

\begin{proposition}\label{prop:Main}
 Let $G$ be a locally pro-$p$ group which contains no proper open centralisers and let $k$ be a field of characteristic $p$. For $C := \ind_{U_\ell}^{G_\ell} (C^\infty(U,k))$ the natural map
\[ \Phi_C : \widehat{k[Z(G)]} \xrightarrow{\;\cong\;} \End_{\Mod(G \times G)}(C)\]
is an isomorphism.
\end{proposition}
\begin{proof}
Suppose given $(\lambda_z)_{z \in \cZ} \in \prod_{z \in \cZ} k[[Z(U)]]$ such that for all $h \in C^\infty(U,k)$, $\lambda_z(h) = 0$ for all but finitely many $z \in \cZ$. Let $f \in C^\infty(Z(U),k)$ be given and choose $h \in C^\infty(U,k)$ extending $f$ using Lemma \ref{lem:ResSurj}. Then $\lambda_z(f) = \lambda_z( \Res^U_{Z(U)}(h) ) = \Res^U_{Z(U)}( \lambda_z(g) ) = 0$ for all but finitely many $z \in \cZ$. In view of condition \eqref{f:imOmegaU}, this shows that in the diagram \eqref{f:GtoZ}, the image of $\Omega_U$ is contained in the image of $\Omega_{Z(U)}$. Write $C_Z := \ind_{Z(U)_\ell}^{Z_\ell} (C^\infty_c(Z,k))$. Since $\Omega_U$ is injective by \eqref{f:OmegaU}, this gives us an injective map $\theta$ which appears as the diagonal arrow in the diagram \eqref{f:GtoZ}:
\[
\xymatrix{ \widehat{k[Z]} \ar[rr]^-{\Phi_C} \ar[d]_{
\Phi_{C_Z}}
 && \End_{\Mod(G \times G)}(C) \ar[r]^{\subseteq} & \End_{\Mod(G \times U)}(C) \ar[dlll]_\theta \ar[d]^{\Omega_U} \\ \End_{\Mod(Z)} (
 C_Z) \ar[rrr]_-{\Omega_{Z(U)}} &&& \prod\limits_{z \in \cZ} k[[Z(U)]]}
\]
and which makes the bottom triangle commutative. Since the rectangle commutes by Lemma \ref{lem:DiagComm} and since $\Omega_{Z(U)}$ is injective, the top triangle commutes. Now, $C_Z$ is isomorphic to $C^\infty_c(Z,k)$ by \eqref{f:trans-ind}, so Prop. \ref{prop:commutative} implies that $\Phi_{C_Z}$ is an isomorphism. Since $\theta$ is injective, we can now conclude that $\Phi_C$ is surjective. But since $C$ is isomorphic to $C^\infty_c(G,k)$ in $\Mod(G)$ by \eqref{f:trans-ind}, $\Phi_C$ is also injective by Prop. \ref{prop:Phi_C}(1), and this completes the proof.
\end{proof}

\begin{theorem}\label{thm:Main}
Let $G$ be a locally pro-$p$ group which contains no proper open centralisers and let $k$ be a field of characteristic $p$. Then the natural map
\[ \Phi : \widehat{k[Z(G)]} \xrightarrow{\;\cong\;} \mathfrak{Z}(G)\]
from Lemma \ref{lem:PhiMap} is an isomorphism.
\end{theorem}
\begin{proof}
By Prop. \ref{prop:Phi_C}(2), we have the commutative diagram
\[\xymatrix{ \widehat{k[Z(G)]} \ar[rr]^{\Phi}\ar[rrd]_{\Phi_C} && \mathfrak{Z}(G) \ar[d]_{\Theta} \\ && \End_{\Mod(G \times G)}(C).
}\]
Now $\Phi_C$ is an isomorphism by Proposition \ref{prop:Main}, and $\Theta$ is injective by the injectivity part of Prop.\ \ref{prop:Pos}. Hence $\Phi$ is an isomorphism.
\end{proof}

\begin{corollary}\label{cor:Main}
Let $G$ be a locally pro-$p$ group which contains no proper open centralisers and let $k$ be a field of characteristic $p$. Then the $k$-algebra homomorphism $\mathfrak{Z}(Z(G)) \xrightarrow{\;\cong\;} \mathfrak{Z}(G)$ from Cor.\ \ref{cor:Z-G} is an isomorphism.
\end{corollary}

Finally, we specialise to the case where $G = \mathbf{G}(\mathfrak{F})$ for some connected smooth algebraic group $\mathbf{G}$ over a local nonarchimedean field $\mathfrak{F}$ of residue characteristic $p$. Certainly then $G$ is locally pro-$p$ (cf.\ \cite{Ser} Thm.\ in \S II.IV.8 and Cor.\ 2 in \S II.IV.9). We quickly remind the reader that the smoothness condition is automatic if $\mathfrak{F}$ has characteristic zero (\cite{DG} II \S 6.1.1). We need the following fact which is known to the experts. We provide a proof since we could not find one in the literature.

\begin{proposition}\label{prop:density}
   Let $X$ be a smooth scheme over the local field $\mathfrak{F}$, and let $Y \subseteq X$ be a Zariski open dense subset. Then $Y(\mathfrak{F})$ is dense in $X(\mathfrak{F})$ w.r.t.\ the valuation topology.
\end{proposition}
\begin{proof}
First we recall the elementary fact that the density of $Y$ in $X$ implies the density of $Y \cap X'$ in any open subset $X'$ of $X$ (check that any nonempty open subset of $X'$ must meet $Y$ non-trivially). Hence we may assume that $X$ is affine.

By passing to one of the finitely many connected components of $X$ we may further assume that $X$ is connected. If $X(\mathfrak{F}) = \emptyset$ then there is nothing to prove. We therefore assume that $X(\mathfrak{F}) \neq \emptyset$. Then $X$ is geometrically connected by \cite{Sta} Lemma 33.7.14. The smoothness finally implies that $X$ is geometrically irreducible. Therefore we may use the implicit function theorem in the form of \cite{GPR} Thm.\ 9.2. It says the following:

Let $x \in X(\mathfrak{F})$ be a point and let $\mathbf{t} = (t_1, \ldots, t_d)$ be a system of local parameters at $x$. Choose a sufficiently small open neighbourhood $x \in U \subseteq X$ on which the $t_i$ are defined. This gives us a morphism $U \rightarrow \mathbf{A}^d$ to the affine space over $\mathfrak{F}$ which sends $x$ to $0$. Then there exists an open neighbourhood $x \in B \subseteq U(\mathfrak{F})$ w.r.t.\ the valuation topology which is mapped by $\mathbf{t}$ homeomorphically onto a polydisk $B_r(0) \subseteq \mathbf{A}^d(\mathfrak{F})$ around $0$ of sufficiently small radius $r$.

We apply this now to a point $x \in X(\mathfrak{F}) \setminus Y(\mathfrak{F)}$. The open neighbourhood $x \in U \subseteq X$ can be chosen to be affine. By the irreducibility of $X$ the Zariski closed subset $U \setminus Y$ of $U$ cannot be equal to $U$. Hence the restriction map $\mathcal{O}(U) \rightarrow \mathcal{O}(U \setminus Y)$ is not injective. We pick a function $f \neq 0$ in the kernel. By possibly making the polydisk $B$ smaller we may view $f = F(t_1,\ldots,t_d)$ as a nonzero formal power series convergent on $B$. The zero set of this power series is nowhere dense in $B$ by the comment after Cor.\ 5 on p.\ 198 in \cite{BGR}. Hence we can find a sequence $(x_n)_{n \in \mathbb{N}}$ of points in $B \subseteq U(\mathfrak{F})$ such that
\begin{equation*}
    \text{$\lim_{n \rightarrow \infty} x_n = x$ w.r.t.\ the valuation topology and $f(x_n) \neq 0$ for any $n \in \mathbb{N}$}.
\end{equation*}
The latter says that each $x_n \in Y(\mathfrak{F})$, because $f$ vanishes on $U \setminus Y$. The former then implies that $x$ lies in the closure of $Y(\mathfrak{F})$ w.r.t.\ the valuation topology.
\end{proof}
\begin{theorem}\label{thm:open-cent} 
Let $\mathbf{G}$ be a connected smooth algebraic group over a local nonarchimedean field $\mathfrak{F}$ of residue characteristic $p$, and let $k$ be a field of characteristic $p$. Then $G = \mathbf{G}(\mathfrak{F})$ contains no proper open centralisers, and we have $ \widehat{k[Z(G)]} \cong \mathfrak{Z}(Z(G)) \cong \mathfrak{Z}(G)$.
\end{theorem}
\begin{proof} Suppose that the centraliser $C_G(y)$ of some $y \in G$ is proper; then $y \notin Z(G)$. Now, $C_G(y) = C_\mathbf{G}(y)(\mathfrak{F})$, where the centraliser $C_\mathbf{G}(y)$ in $\mathbf{G}$ of the closed point $y$ is a closed algebraic $\mathfrak{F}$-subgroup of $\mathbf{G}$ by \cite{DG} II \S 1 Cor.\ 3.7. Since $y \notin Z(G)$, $\mathbf{G} \setminus C_\mathbf{G}(y)$ is a nonempty Zariski open subset of $\mathbf{G}$. Since $\mathbf{G}$ is irreducible by \cite{DG} II \S 5 Thm.\ 1.1 it also is Zariski dense in $\mathbf{G}$. So we may apply Prop.\ \ref{prop:density} with $X = \mathbf{G}$ and $Y = \mathbf{G} \setminus C_\mathbf{G}(y)$ to obtain that $G \setminus C_G(y) = (\mathbf{G} \setminus C_\mathbf{G}(y))(\mathfrak{F})$ is dense in $\mathbf{G}(\mathfrak{F}) = G$. Hence $C_G(y)$ cannot be open in $G$.

The second part of the assertion now follows from Thm.\ \ref{thm:Main} and Cor.\ \ref{cor:Main}.
\end{proof}

\begin{corollary}\label{cor:semisimple}
  Suppose that the connected smooth algebraic group $\mathbf{G}$ is semisimple and that $k$ has characteristic $p$; then the Bernstein center of $\Mod(G)$ is $\mathfrak{Z}(G) = k[Z(G)]$.
\end{corollary}
\begin{proof}
Under the present assumption $Z(G)$ is finite.
\end{proof}
\section{Appendix: Positselski's Theorem}\label{sec:PosThmNewPf}

We give a proof of Prop.\ \ref{prop:Pos} which is more direct than the one in \cite{Pos}.

For the \textbf{injectivity} of $\Theta$ let $(T_V)_V \in \mathfrak{A}(G)$ such that $T_{C_c^\infty(G,k)} = 0$. Since $C_c(G/U,k)$, for any compact open subgroup $U \subseteq G$, is a $G_\ell$-subrepresentation of $C_c^\infty(G,k)$ it follows that $T_{C_c^\infty(G/U,k)} = 0$. Now consider any object $V$ in $\Mod(G)$. Any vector $v \in V$ is fixed by some compact open subgroup $U_v \subseteq G$. By Frobenius reciprocity we obtain a $G$-equivariant map $\rho_v : C_c(G/U_v,k) \rightarrow V$ sending the characteristic function of $U_v$ to $v$. We deduce that $T_{\im(\rho_v)} = 0$. By varying the vector $v$ finally obtain that $T_V = 0$.

The argument for the \textbf{surjectivity} of $\Theta$ is slightly more involved.

\begin{lemma}\label{lem:A1}
  Inside the algebra $\End_k(C_c^\infty(G,k))$ the subalgebras $\End_{\Mod(G_\ell)}(C_c^\infty(G,k))$ and $\End_{\Mod(G_r)}(C_c^\infty(G,k))$ are mutually centralisers of each other.
\end{lemma}
\begin{proof}
Since $k[G_r] \subseteq \End_{\Mod(G_\ell)}(C_c^\infty(G,k))$ and $k[G_\ell] \subseteq \End_{\Mod(G_r)}(C_c^\infty(G,k))$ it is immediate that the centraliser of one algebra is contained in the other algebra. For the opposite inclusions we have to show that $\beta \circ \alpha = \alpha \circ \beta$ for any $\alpha \in \End_{\Mod(G_\ell)}(C_c^\infty(G,k))$ and $\beta \in \End_{\Mod(G_r)}(C_c^\infty(G,k))$. Since
\begin{equation*}
   \bigcup_U C_c(U\backslash G,k) = C_c^\infty(G,k) = \bigcup_U C_c(G/U,k) \ ,
\end{equation*}
where $U$ runs over all compact open subgroups of $G$, both homomorphisms $\alpha$ and $\beta$ are completely determined by their values on the characteristic functions $\chara_U$. Moreover, $\alpha(\chara_U) \in C_c(U\backslash G,k)$ and $\beta(\chara_U) \in C_c(G/U,k)$. We have to check that $\beta \circ \alpha (\chara_{xU}) = \alpha \circ \beta (\chara_{xU})$ for any $U$ and any $x \in G$. Let $\alpha(\chara_U) = \sum_i c_i \chara_{U g_i}$ and $\beta(\chara_{xUx^{-1}}) = \sum_j d_j \chara_{h_j xUx^{-1}}$ for finitely many constants $c_i, d_j \in k$ and elements $g_i, h_j \in G$ (depending on $x$ and $U$). We compute
\begin{align*}
  \beta \circ \alpha (\chara_{xU}) & = \beta( \sum_i c_i \chara_{xU g_i}) = \sum_i c_i \beta(\chara_{xU g_i}) = \sum_i c_i \beta(\chara_{xUx^{-1} x g_i}) \\
        & = \sum_i c_i \beta({^{( x g_i)^{-1}_r}\chara}_{xUx^{-1}}) = \sum_i c_i {^{( x g_i)^{-1}_r}\beta}(\chara_{xUx^{-1}}) =   \\
        & = \sum_{i, j} c_i d_j {^{( x g_i)^{-1}_r} \chara}_{h_j xUx^{-1}} = \sum_{i, j} c_i d_j \chara_{h_j xUg_i}
\end{align*}
using that $\alpha$ is $G_\ell$-equivariant, resp.\ $\beta$ is $G_r$-equivariant, in the first, resp.\ fifth, equality. Correspondingly
\begin{align*}
  \alpha \circ \beta (\chara_{xU}) & = \alpha(\beta({^{x^{-1}_r} \chara}_{xUx^{-1}})) = \alpha({^{x^{-1}_r} \beta}(\chara_{xUx^{-1}})) = \alpha( {^{x^{-1}_r} \sum_j} d_j \chara_{h_jxUx^{-1}} )   \\
       & = \sum_j d_j \alpha( \chara_{h_jxU} ) = \sum_{i, j} c_i d_j \chara_{h_jxUg_i} \ . \qedhere
\end{align*}
\end{proof}

\begin{lemma}\label{lem:A2}
  Suppose that $U_i$, for $i = 1, 2$, are compact open subgroups of $G$; any map $\rho \in \Hom_{\Mod(G)}(C_c(G/U_1,k), C_c(G/U_2,k))$ is the restriction of a map $\tilde{\rho} \in \End_{\Mod(G_\ell)}(C_c^\infty(G,k))$.
\end{lemma}
\begin{proof}
By Frobenius reciprocity we have
\begin{equation*}
  \Hom_{\Mod(G)}(C_c(G/U_1,k), C_c(G/U_2,k)) = C_c(U_1 \backslash G/U_2,k)
\end{equation*}
with $\rho$ corresponding to $\rho(\chara_{U_1})$ and
\begin{align*}
  \End_{\Mod(G_\ell)}(C_c^\infty(G,k)) & = \Hom_{\Mod(G_\ell)}( \bigcup_{U \subseteq U_1} C_c(G/U,k), C_c^\infty(G,k))   \\
                        & = \varprojlim_{U \subseteq U_1} \Hom_{\Mod(G_\ell)}(C_c(G/U,k), C_c^\infty(G,k))   \\
                        & = \varprojlim_{U \subseteq U_1} C_c(U \backslash G,k)
\end{align*}
with $\tilde{\rho}$ corresponding to $(\tilde{\rho}(\chara_U))_U$. The transition map $C_c(U' \backslash G,k) \rightarrow C_c(U \backslash G,k)$, for $U' \subseteq U$, in the projective system is given by $\chara_{U' g} \mapsto \chara_{Ug}$. Suppose that $\rho(\chara_{U_1}) = \sum_i c_i \chara_{U_1 g_i}$. Then $(\sum_i c_i \chara_{U g_i})_U$ defines a $\tilde{\rho}$ as in the assertion which restricts to $\rho$.
\end{proof}

We now fix a $T \in \End_{\Mod(G_r)}(C_c^\infty(G,k))$, and we have to construct a $(T_V)_V \in \mathfrak{A}(G)$ such that $T_{C_c^\infty(G,k)} = T$. Of course, we turn the latter condition into the definition $T_{C_c^\infty(G,k)} := T$ and extend it in several steps to all of $\Mod(G)$.

\textit{Step 1:} Let $\mathfrak{C}_0$ denote the full subcategory of $\Mod(G)$ with the single object $C_c^\infty(G,k)$. Then Lemma \ref{lem:A1} tells us that $T_{C_c^\infty(G,k)} = T$ is a natural transformation of the forgetful functor $\mathfrak{C}_0 \rightarrow \Vect$.

\textit{Step 2:} Let $\mathfrak{C}_1$ denote the full subcategory of $\Mod(G)$ with objects representations of the form $C_c(G/U,k)$ for $U \subseteq G$ some compact open subgroup. We have $C_c(G/U,k) = C_c^\infty(G,k)^{U_r} \subseteq C_c^\infty(G,k)$. The map $T$ being $G_r$-equivariant it respects the subrepresentation $C_c(G/U,k)$, and we may define $T_{C_c(G/U,k)} := T | C_c(G/U,k)$. It follows from Lemma \ref{lem:A2} and Step 1 that this defines a natural transformation of the forgetful functor $\mathfrak{C}_1 \rightarrow \Vect$.

\textit{Step 3:} Let $\mathfrak{C}$ denote the full subcategory of $\Mod(G)$ with objects arbitrary direct sums of objects in $\mathfrak{C}_1$. If $V = \oplus_{i \in I} V_i$ with $V_i$ in $\mathfrak{C}_1$ then we simply define $T_V := \oplus_{i \in I} T_{V_i}$. Let $V' = \oplus_{j \in J} V'_j$ be another object in $\mathfrak{C}$ with $V'_j$ in $\mathfrak{C}_1$ and let $\gamma : V \rightarrow V'$ be a homomorphism in $\Mod(G)$. There are homomorphisms $\gamma_{i,j} : V_i \rightarrow V'_j$ in $\Mod(G)$ such that $\gamma((v_i)_i) = (\sum_i \gamma_{i,j}(v_i))_j$. Using Step 2 in the fourth equality we compute
\begin{align*}
  T_{V'} \circ \gamma ((v_i)_i) & = T_{V'}((\sum_i \gamma_{i,j}(v_i))_j) = ( T_{V'_j}(\sum_i \gamma_{i,j}(v_i)))_j = ( \sum_i T_{V'_j}(\gamma_{i,j}(v_i)))_j   \\
               & = ( \sum_i \gamma_{i,j}(T_{V_i}(v_i)))_j = \gamma ((T_{V_i}(v_i))_i) = \gamma \circ T_V ((v_i)_i) \ .
\end{align*}
This shows that our definition extends $T$ further to a natural transformation of the forgetful functor $\mathfrak{C} \rightarrow \Vect$.

\textit{Step 4:} In the proof of injectivity we have seen already that for any $V$ in $\Mod(G)$ there is a surjective homomorphism $C \twoheadrightarrow V$ in $\Mod(G)$ with $C$ in $\mathfrak{C}$.

\textit{Step 5:} Consider a surjective homomorphism $\beta : C \twoheadrightarrow V$  as well as a surjective homomorphism $\gamma : C' \twoheadrightarrow \ker(\beta)$ as in Step 4. By Step 3 we then have the commutative diagram
\begin{equation*}
  \xymatrix{
    C' \ar[d]_{T_{C'}} \ar@{->>}[r] & \ker(\beta)  \ar[r]^-{\subseteq} & C \ar[d]^{T_C} \\
    C' \ar@{->>}[r] & \ker(\beta) \ar[r]^-{\subseteq} & C.   }
\end{equation*}
This shows that $T_C$ respects $\ker(\beta)$ and induces therefore a map $T_V$ on $V$. We first check that $T_V$ is independent of the choices. For this let $\beta_i : C_i \twoheadrightarrow V$, for $i = 1, 2$, be two surjective homomorphisms as in Step 4. At first we consider the case that there is a homomorphism $\gamma : C_1 \rightarrow C_2$ in $\Mod(G)$ such that $\beta_2 \circ \gamma = \beta_1$. Let $\tau_i$ be the linear endomorphism of $V$ induced by $T_{C_i}$. Given a $v \in V$ we choose a preimage $x \in C_1$ such that $\beta_1(x) = v$. Using Step 3 in the third equality we compute
\begin{equation*}
  \tau_1(v) = \beta_1(T_{C_1}(x)) = \beta_2 \circ \gamma \circ T_{C_1}(x) = \beta_2 \circ T_{C_2} \circ \gamma(x) = \tau_2(\beta_2(\gamma(x))) = \tau_2(\beta_1(x)) = \tau_2(v)  .
\end{equation*}
In general without the presence of a $\gamma$ we consider the commutative diagram
\begin{equation*}
  \xymatrix{
  C_1 \ar[r] \ar[dr]_{\beta_1} & C_1 \oplus C_2 \ar[d]^(0.4){\beta_1 + \beta_2} & C_2 \ar[l] \ar[dl]^{\beta}  \\
                & V .                 }
\end{equation*}
with the horizontal maps being the canonical ones and apply the previous case to both halves.

\textit{Step 6:} Finally we have to show that the thus defined $T_V$ constitute a natural transformation of the forgetful functor $\Mod(G) \rightarrow \Vect$, i.e., we have to check that, for any homomorphism $\gamma : V_1 \rightarrow V_2$ in $\Mod(G)$ the diagram
\begin{equation*}
  \xymatrix{
    V_1 \ar[d]_{T_{V_1}} \ar[r]^{\gamma} & V_2 \ar[d]^{T_{V_2}} \\
    V_1 \ar[r]^{\gamma} & V_2   }
\end{equation*}
is commutative. We choose two surjective homomorphisms $\beta_i : C_i \twoheadrightarrow V_i$, for $i = 1, 2$, as in Step 4 and obtain the commutative diagram
\begin{equation*}
  \xymatrix{
    C_1 \ar@{->>}[d]_{\beta_1} \ar[r]^-{\subseteq} & C_1 \oplus C_2 \ar@{->>}[d]^{\gamma \circ \beta_1 + \beta_2} \\
    V_1 \ar[r]^{\gamma} & V_2 .   }
\end{equation*}
It then follows from Step 3 that $\gamma \circ T_{V_1} = T_{V_2} \circ \gamma$.

\end{document}